\documentclass[pdflatex,a4paper,10pt]{article}
\usepackage[numbers,longnamesfirst,square]{natbib}
\usepackage{amsmath,amsfonts,amssymb,latexsym,wasysym,mathrsfs}
\usepackage{amsthm}
\usepackage{pstricks,a4wide,setspace} 
\usepackage[english]{babel}

\theoremstyle{plain}
\newtheorem{theorem}{Theorem}[section]
\newtheorem{proposition}{Proposition}[section]
\newtheorem{lemma}{Lemma}[section]

\theoremstyle{definition}
\newtheorem{definition}{Definition}[section]

\theoremstyle{remark}
\newtheorem{remark}{Remark}[section]

\theoremstyle{example}
\newtheorem*{example}{Example}%[section]

\author{Francesca Collet\\
{\small Dipartimento di Matematica} \\
{\small Alma Mater Studiorum Universit\`a di Bologna}\\
{\small Piazza di Porta San Donato 5; 40126 - Bologna, Italy} \\
{\small e-mail: francesca.collet@unibo.it}\\[0.5cm]
Fabrizio Leisen\\
{\small Departamento de Estad\'istica} \\
{\small Universidad Carlos III de Madrid}\\
{\small Calle Madrid 126; 28903 - Getafe (Madrid), Spain} \\
{\small e-mail: fabrizio.leisen@gmail.com}\\[0.5cm]
Fabio Spizzichino\thanks{Corresponding author.}\\
{\small Dipartimento di Matematica} \\
{\small Sapienza Universit\`a di Roma}\\
{\small Piazzale Aldo Moro 5; 00185 - Roma, Italy } \\
{\small e-mail: fabio.spizzichino@uniroma1.it}\\[0.5cm]
Florentina Suter\\
{\small Facultatea de Matematica si Informatica} \\
{\small Universitatea din Bucuresti}\\
{\small Str. Academiei, 14; 010014 - Bucuresti, Romania}\\
{\small e-mail: florentina.suter@fmi.unibuc.ro}}

\title{EXCHANGEABLE OCCUPANCY MODELS AND DISCRETE PROCESSES WITH THE GENERALIZED UNIFORM ORDER STATISTICS PROPERTY}
\date{\today}

\begin{document}

\maketitle

\begin{abstract}
% text of abstract goes here!
This work focuses on Exchangeable Occupancy Models (EOM) and their relations with the Uniform Order Statistics Property (UOSP) for point processes in discrete time. 
As our main purpose, we show how definitions and results presented in \citet{ShSpSu04} 
can be unified and generalized in the frame of occupancy models.
We first show some general facts about EOM's. Then we introduce a class of EOM's, called $\mathcal{M}^{(a)}$-models, and a concept of generalized Uniform Order Statistics Property in discrete time. For processes with this property, we prove a general characterization result in terms of $\mathcal{M}^{(a)}$-models. Our interest is also focused on properties of closure w.r.t. some natural transformations of EOM's.\\

\vspace{0.3cm}

\noindent \textbf{Keywords:} $\mathcal{M}_{n,r}^{(a)}$-Models, Random Sampling of Arrivals, Transformations of Occupancy Models,  Closure Properties

\vspace{0.3cm}

\noindent \textbf{AMS 2010 Classification:} 62G30, 60G09, 60G55
%62G30 Order Statistics
%60G09 Exchangeability
%60G55 Point Processes
\end{abstract}

\doublespacing

\section{Introduction}
%=============================================================================

The so called {\it occupancy distributions} give rise, as well known, to a class of multivariate models useful in the description of randomized phenomena. The name ``occupancy'' comes from the interpretation in terms of particles that are randomly distributed among several cells. In particular, three classical examples, related to as many well-known physical systems, belong to this class: Maxwell-Boltzmann, Bose-Einstein and Fermi-Dirac model, see \citet{Fel68}. In the Maxwell-Boltzmann (MB) statistics the capacity of each cell is unlimited, and the particles are distinguishable. In the Bose-Einstein (BE) statistics, the capacity of each cell is unlimited but the particles are indistinguishable. In the Fermi Dirac (FD) statistics, the particles are indistinguishable, but cells can only hold a maximum of one particle. All these three statistics assume that the cells are distinguishable.
These models are attractive for many reasons. First of all they have wide range of applications in Sciences, Engineering and also in Statistics, as pointed out by \citeauthor{Cha05} \citep[Chapters 4 and 5]{Cha05} and \citet{GaDr2011}. Moreover \citet{Mahmoud} provided an interpretation of occupancy distributions in terms of P\'olya Urns, which are very flexible and applicable to problems arising in various areas; e.g. Clinical Trials (see \citet{CrLei2008} for some references), Economics (see \citet{Aruka11}) and Computer Science (see \citet{ShKoCh11}). From a probabilistic and combinatoric point of view, the three fundamental models (MB, BE and FD) have many interesting properties. Indeed they are, in particular, \emph{exchangeable} and this is a basic remark for our aims. This paper is in fact concentrated on the theme of Exchangeable Occupancy Models (EOM) and their relations with the Uniform Order Statistics Property (UOSP) of counting processes in discrete time. As one main purpose of ours, we show that some notions and results given in \citet*{ShSpSu04,ShSpSu08} admit completely natural generalizations in the frame of EOM's.

% We start proving some general properties of the EOM's. In particular, we show that the exchangeability property is maintained under remarkable types of transformations of occupancy models. Then, we introduce the notion of $\mathcal{M}^{(a)}$-models, a relevant sub-class of EOM's that turns out to have an important role in our derivations. In the final part of the paper, we study discrete-time counting processes satisfying the generalized UOSP; in other words, for which the joint distribution of the jump amounts is distributed according to a $\mathcal{M}^{(a)}$-model, conditionally on a fixed number of arrivals up to time $t$. 
 
After appropriate preliminaries, we will consider some general properties of the EOM's. For our purposes, we then introduce the notion of $\mathcal{M}^{(a)}$-models, a relevant sub-class of EOM's that turns out to have an important role in our derivations. In the final part of the paper, we will introduce and analyze discrete-time generalized UOSP. Such a property, will be defined by imposing the form of an  
$\mathcal{M}^{(a)}$-model to the joint distribution of the process jump amounts, conditionally on a fixed number of arrivals up to any given time $t$. In particular, for these processes, several characterizations are proved. \\
More in detail, the outline of the paper is as follows. In Section~\ref{Sect:Review}, we will start with recalling some basic notions and fixing some necessary notation. This will allow us to better explain the motivations of our 
work. Our central results will be presented in the Section~\ref{Sect:M-models}, where we define the class of occupancy $\mathcal{M}^{(a)}$-models and we show natural extensions of the results proved in \citet{ShSpSu04,ShSpSu08}. Section~\ref{Sect:EOMandUOSP} is devoted to some preliminary arguments about occupancy models and Section \ref{Sect:EOM} presents some specific aspects of the EOM's. In Section~\ref{Sect:Transformations} some additional properties of occupancy models are given. In particular, we analyze closure properties under remarkable types of transformations of occupancy models.

%=============================================================================
\section{Brief review and motivations}\label{Sect:Review}
%=============================================================================

Our work arises naturally as an attempt to generalize the definitions of UOSP presented in \citet{ShSpSu04,ShSpSu08}. For sake of readability and to make the setting clear, first we briefly summarize main facts of interest therein contained. Secondly, we explain the role played by our paper in this context and the connection with exchangeable occupancy models. \\

We start with some notation. We need to consider two distinct kinds of discrete-time, discrete-space counting processes: with unit and multiple jumps respectively.

\begin{enumerate}
\item Let $\{M_t\}_{t = 0,1,\dots}$ be a discrete-time counting process with \emph{unit jumps}, namely $M_{t+1} - M_t \leq 1$. Assume that $M_0=0$ and $P \{ \lim_{t \to +\infty} M_t = +\infty \} = 1$. \\
Moreover, let $T_1, T_2, \dots$ denote the arrival times of the process; in other words, 
\[
T_k = t \Longleftrightarrow M_{t-1}=k-1 \mbox{ and }  M_t=k \quad (1 \leq k \leq t) \,.
\] 

\item Let $\{N_t\}_{t = 0,1,\dots}$ be a discrete-time counting process with \emph{jump amounts} $\{ J_k\}_{k = 0,1,\dots}$. Assume that $P \{ \lim_{t \to +\infty} N_t = +\infty \} = 1$. \\
Moreover, let $T_1, T_2, \dots$ denote the arrival times of the process; in other words, 
 \[
T_k = \inf \{t \geq 0 : N_t \geq k\} \quad (1 \leq k \leq t) \,.
\] 
\end{enumerate}

Inspired by the UOSP for continuous-time processes, \citet{HuSh94} introduced a corresponding property for discrete-time counting processes with unit step jumps. 

\begin{definition}
$\{M_t\}_{t = 0,1,\dots}$ satisfies the UOSP$(<)$ if, for any $0 < t_1 < t_2 < \cdots < t_k \leq t$, we have
\[
P \{ T_1=t_1, T_2=t_2, \dots, T_k=t_k \vert M_t=k \} = \binom{t}{k}^{-1} \,.
\]
\end{definition}

Successively, in \citet{ShSpSu04,ShSpSu08} an extension of UOSP to processes with multiple jumps is given.

\begin{definition}
$\{N_t\}_{t = 0,1,\dots}$ satisfies the
\begin{itemize}
\item UOSP$(\leq_1)$
if, for any $0 \leq t_1 \leq t_2 \leq \cdots \leq t_k \leq t$, we have
\[
P \{ T_1=t_1, T_2=t_2, \dots, T_k=t_k \vert N_t=k \} = \frac{k!}{j_0! j_1! \cdots j_t!} \left( \frac{1}{t+1} \right)^k \,,
\]
where, for $\ell \in \{ 0,1,\dots,t\}$, $j_\ell$ is the number of values in $\{t_1,t_2,\dots,t_k\}$ that are equal to $\ell$.
\item UOSP$(\leq_2)$
if, for any $0 \leq t_1 \leq t_2 \leq \cdots \leq t_k \leq t$, we have
\[
P \{ T_1=t_1, T_2=t_2, \dots, T_k=t_k \vert N_t=k \} = \binom{t+k}{k}^{-1} \,.
\]
\end{itemize}
\end{definition}

The UOSP's give important information about the possibility for a counting process of being a mixed geometric or a mixed uniform sample process. Recall that a counting process $\{A_t\}_{t =0,1,\dots}$ is called mixed geometric if there exists a random variable $\Theta \in (0,1)$ such that, given $\Theta=\theta$, the inter-arrival times are i.i.d. geometric with parameter $\theta$. Whereas $\{A_t\}_{t=0,1,\dots,\tau}$, with $\tau < \infty$, is a mixed uniform sample process if there exists a positive integer-valued random variable $\Theta$ such that, for $t=0,1,\dots,\tau$,
\[
A_t = \sum_{k=0}^\Theta \mathbf{1}_{\{U_k \leq t\}} \text{ in distribution} \,,
\] 
where $U_1, \dots, U_\Theta$ are independent uniform random variables on $\{0,1,\dots,\tau\}$, independent on $\Theta$. \\
We are now ready to state the following characterization result.

\begin{theorem}\label{Thm:UOSP's}
\begin{enumerate}
\item (\citet{HuSh94}) $\{M_t\}_{t = 0,1,\dots}$ satisfies the UOSP$(<)$ if and only if it is a mixed geometric process.
\item (\citet{ShSpSu04}) $\{N_t\}_{t = 0,1,\dots}$ satisfies the UOSP$(\leq_1)$ if and only if $\{N_t\}_{t=0,1,\dots,\tau}$ is a mixed uniform sample process for every $\tau \geq 0$ such that $P \{ N_\tau < N_s \mbox{ for some } s > \tau \}$.
\end{enumerate}
\end{theorem}

The result concerning the process with multiple jumps has basically two limitations compared to the one for the process with unit jumps: the validity of the equivalence is restricted to a finite set of times and the counting process must be bounded. Besides, the characterization of UOSP$(\leq_2)$, that has been presented in \citet{ShSpSu04,ShSpSu08}, has a different format from those given in Theorem
\ref{Thm:UOSP's}. Namely, it has been shown that a process has the UOSP$(\leq_2)$ if and only if the epoch times $T_1, \dots, T_k$ have a joint $\ell_\infty^{\leq}$-spherical density.\\ 

To bridge the gap, we suggest a modification of UOS and mixed geometric properties whose consequence is twofold. On one hand, we generalize the UOSP$(<)$ and UOSP$(\leq_{1,2})$ introduced in \citet{HuSh94,ShSpSu04,ShSpSu08}. On the other, we create an extended framework where $\{ N_{t} \}_{t =0,1,\dots}$ obeys a result analogous to the one fulfilled by $\{ M_{t} \}_{t =0,1,\dots}$; in other words, in which having the generalized UOSP is equivalent to being mixed geometric according to an appropriate definition. \\
The key approach to suitably deduce the new definitions of UOSP and ``mixed geometric'' is to read 
the conditional distribution of the arrival times, given a fixed number of arrivals, as an element in the class of the occupancy $\mathcal{M}^{(a)}$-models. It turns out that, in terms of \emph{exchangeable occupancy models}, 
%are the natural view point for studying counting processes, since this \emph{interpretation} allows us to 
we can create a \emph{unified framework} where processes with unit and multiple jumps satisfy analogous properties.

%=============================================================================
\section{Occupancy models and order statistics of discrete variables}\label{Sect:EOMandUOSP}
%=============================================================================

For fixed $n=2,3,\dots$ and $r=1,2,\dots$, let $A_{n,r}$ be the set defined by
\[
A_{n,r} := \left\{ \mathbf{x} \equiv \left(  x_{1},\dots,x_{n} \right) : x_{j} = 0, 1, \dots, r  \mbox{ and } \sum_{j=1}^{n} x_{j}=r \right\} \,.
\]
%Any probability distribution on $A_{n,r}$ can be seen as an \emph{occupancy model}.

As it is well-known (see e.g. \citet{Fel68}) the cardinality of $A_{n,r}$ is
\[
\left\vert A_{n,r} \right\vert = \binom{n+r-1}{n-1} \,.
\]
Starting from the classical scheme of $r$ \emph{particles} that are distributed stochastically into $n$ \emph{cells}, we consider the random vector $\mathbf{X} \equiv (X_{1}, X_{2}, \dots, X_{n})$ where, for $j=1, 2, \dots, n$, $X_{j}$ is the $\{ 0, 1, \dots, r \}$-valued random variable that counts the number of particles fallen in $j$-th cell. $X_{1}, X_{2}, \dots, X_{n}$ are called
\emph{occupancy numbers} and the joint distribution of $(X_{1}, X_{2}, \dots, X_{n})$, describing the probabilistic mechanism of assignment of the particles to the cells, is called an \emph{occupancy model}. Since the vector $(X_{1}, X_{2}, \dots, X_{n})$ takes its values in the set $A_{n,r}$, an \emph{occupancy model} is then a probability distribution on $A_{n,r}$. \\

Let now $B_{r,n}$ denote the set
\[
B_{r,n} := \left\{ \mathbf{u} \equiv (u_{1}, u_{2}, \dots, u_{r}) : u_{i} =1, 2, \dots, n \mbox{ and } u_{1} \leq
u_{2}\leq \cdots \leq u_{r} \right\} 
\]
and consider the mapping $\varphi: B_{r,n} \longrightarrow A_{n,r}$ defined as
\[
\varphi (\mathbf{u}) = (\varphi_{1} (\mathbf{u}), \varphi_{2} (\mathbf{u}), \dots, \varphi_{n}(\mathbf{u})) \,,
\]
with
\[
\varphi_{j} (\mathbf{u}) = \sum_{i=1}^{r} \mathbf{1}_{\{u_{i}=j\}}\,, \quad \mbox{ for } j = 1, 2, \dots, n.
\]
This is a one-to-one correspondence and then, for the cardinality of $B_{r,n}$, we have
\[
\left\vert B_{r,n}\right\vert = \left\vert A_{n,r} \right\vert =\binom{n+r-1}{n-1}.
\]
As to $\psi = \varphi^{-1}: A_{n,r} \longrightarrow B_{r,n}$, we can write
\[
\psi (\mathbf{x}) = (\psi_{1} (\mathbf{x}), \psi_{2} (\mathbf{x}), \dots, \psi_{r} (\mathbf{x})) \,,
\]
with
\[
\psi_{i} (\mathbf{x}) = \min \left\{ s \left\vert \sum_{j=1}^{s} x_{j} \geq i \right. \right\} \,, \quad \mbox{ for } i = 1, 2, \dots, r.
\]

We can thus consider the random vector $\mathbf{U} \equiv (U_{1}, \dots, U_{r})$ defined by
\begin{equation}\label{DefinVarU}%
\mathbf{U} = \psi\left( \mathbf{X} \right).
\end{equation}

For $\left( u_{1}, \dots, u_{r} \right) \in B_{r,n}$, one then has
\begin{equation*}\label{DisrtProbUvsX}
P \{U_{1}=u_{1}, \dots, U_{r}=u_{r} \} = P \{ X_{1} = \varphi_{1} (\mathbf{u}), \dots, X_{n} = \varphi_{n} (\mathbf{u}) \}.
\end{equation*}

Let $\mathcal{P} \left( A_{n,r} \right)$ and $\mathcal{P} \left( B_{r,n} \right)$ respectively denote the family of probability
distributions on $A_{n,r}$ and the family of probability distributions on $B_{r,n}$. In view of the above one-to-one correspondence between $A_{n,r}$ and $B_{r,n}$, we can consider the induced (one-to-one) correspondence between $\mathcal{P} \left( A_{n,r} \right)$ and $\mathcal{P} \left( B_{r,n} \right)$. More precisely, we consider the mappings
$\Phi$ and $\Psi = \Phi^{-1}$ defined by
\[
\Psi(Q) (\mathbf{x}) = Q[\psi (\mathbf{x})] \mbox{ and } \Phi(P) (\mathbf{u}) = P[\varphi (\mathbf{u})] \,,
\]
where $P \in \mathcal{P} \left( A_{n,r} \right)$ and $Q \in \mathcal{P} \left( B_{r,n} \right)$.

\begin{remark}
Note that $P$ is the uniform distribution on $A_{n,r}$ (i.e., the Bose-Einstein model, see also Section~\ref{Sect:EOM}) if and only if $Q=\Phi(P)$ is the uniform distribution on $B_{r,n}$.
\end{remark}
\bigskip

Consider now $r$ exchangeable random variables $Y_{1}, Y_{2}, \dots, Y_{r}$, which take values in $\{ 1, 2, \dots, n \}$. To these random variables we can associate a vector of occupancy numbers by introducing the random variables $X_{1}, X_{2}, \dots, X_{n}$ defined by
\[
X_{j} = \sum_{h=1}^{r} \mathbf{1}_{\{Y_{h} = j\}}\,, \quad \mbox{ for } j = 1, 2, \dots, n.
\]
The set of the possible values taken by $\mathbf{Y} \equiv (Y_1, \dots, Y_r)$ is then $D_{r,n}:=\{1,2,\dots,n\}^{r}$. \\
We will also write $\mathbf{X}=\tilde{\varphi}\left(  \mathbf{Y}\right)$ or %
\begin{equation}\label{XInFunzdiY}
X_{j} = \tilde{\varphi}_{j} (Y_{1}, Y_{2}, \dots, Y_{r}), \quad \mbox{for } j=1,\dots,n
\end{equation}
where $\tilde{\varphi}:D_{r,n}\longrightarrow A_{n,r}$ with

\begin{equation}\label{tilde_varphi_function}
\tilde{\varphi}_{j} (y_{1}, y_{2}, \dots, y_{r}) = \sum_{h=1}^{r} \mathbf{1}_{ \{y_{h} = j \}}\,, \quad \mbox{for } \mathbf{y}\in D_{r,n} \mbox{ and } j = 1, 2, \dots, n.
\end{equation}

It can be immediately seen that the probability distribution of $\left( Y_{1}, Y_{2}, \dots, Y_{r}\right)$ is uniquely determined by the probability distribution of $\left( X_{1}, X_{2}, \dots, X_{n} \right)$ and vice versa.

In fact, the following relationships hold
\begin{equation}\label{XvsY}
P \{ Y_{1} = y_{1}, Y_{2} = y_{2}, \dots, Y_{r} = y_{r} \} = \frac{P \{ X_{1} = \tilde{\varphi}_{1} (\mathbf{y}), X_{2} = \tilde{\varphi}_{2} (\mathbf{y}), \dots, X_{n} = \tilde{\varphi}_{n} (\mathbf{y}) \}}{\binom{r}{\tilde{\varphi}_{1} (\mathbf{y}) \tilde{\varphi}_{2} (\mathbf{y}) \cdots \tilde{\varphi}_{n} (\mathbf{y})}}
\end{equation}
and
\begin{multline}\label{YvsX}
P \{ X_{1} = x_{1}, X_{2} = x_{2}, \dots, X_{n} = x_{n} \} =\\
= \binom{r}{x_{1} \cdots x_{n}} P \{Y_{1} = \psi_{1} (\mathbf{x}), Y_{2} = \psi_{2} (\mathbf{x}), \dots, Y_{r} = \psi_{r} (\mathbf{x}) \}
\end{multline}

We note that $\tilde{\varphi}:D_{r,n} \longrightarrow A_{n,r}$ is different from the previous transformation $\varphi$, which is defined on $B_{r,n}\subset D_{r,n}$ and is bijective. However, $\tilde{\varphi}(\mathbf{y})=\varphi(\mathbf{y})$ if $\mathbf{y} \in B_{r,n}$ and, for any $\mathbf{y} \in D_{r,n}$, we have 
\[
\tilde{\varphi}(\mathbf{y})=\tilde{\varphi}(y_{(1)},\dots,y_{(r)})=\varphi(y_{(1)},\dots,y_{(r)}),
\]
where $(y_{(1)},\dots,y_{(r)})$ is the vector of the coordinates of $\mathbf{y}$ rearranged in increasing order.\\

We consider now the vector $\mathbf{Y}_{\left( \cdot \right)} \equiv \left( Y_{(1)}, \dots ,Y_{(r)} \right)$ of the order statistics of the exchangeable vector $(Y_{1}, Y_{2}, \dots, Y_{r})$. The set of values taken by $\mathbf{Y}_{\left( \cdot \right)}$ is $B_{r,n}$. To a probability distribution of $(Y_{1}, Y_{2}, \dots, Y_{r})$ it corresponds one and only one probability distribution of $\mathbf{Y}_{\left( \cdot \right)}$ and, for $\mathbf{u} \in B_{r,n}$, we can write

\[
P \{ Y_{1} = u_{1}, \dots, Y_{r} = u_{r} \} = \frac{P \{ Y_{(1)} = u_{1}, \dots, Y_{(r)} = u_{r} \}}{\binom{r}{\varphi_{1} (\mathbf{u}) \cdots \varphi_{n} (\mathbf{u})}}.
\]

Furthermore,
\[
\varphi \left( \mathbf{Y}_{( \cdot )} \right) = \tilde{\varphi} \left( \mathbf{Y}_{( \cdot )} \right) = \tilde{\varphi} \left( \mathbf{Y} \right) = \mathbf{X} \,,
\]
and thus
\[
\mathbf{Y}_{( \cdot )} = \varphi^{-1} \left( \mathbf{X} \right) = \psi \left( \mathbf{X} \right).
\]

We can then summarize the arguments above as follows.

\begin{proposition}\label{Prop:order}
Let $X_{1}, X_{2}, \dots, X_{n}$ be random variables such that the support of their joint distribution is $A_{n,r}$, for some $r \in \mathbb{N}$. Moreover, let $Y_{1}, Y_{2}, \dots, Y_{r}$ be exchangeable, $\{ 1, 2, \dots, n \}$-valued random  variables such that \eqref{XInFunzdiY} holds.\\
Then, the random variables $U_{1}, U_2, \dots, U_{r}$, defined by equation \eqref{DefinVarU}, are the order statistics of the vector $(Y_1, \dots, Y_r)$. 
\end{proposition}

%=============================================================================
\section{Exchangeable Occupancy Models}\label{Sect:EOM}
%=============================================================================

Often relevant occupancy models are such that the random variables $X_{1}, \dots, X_{n}$ are exchangeable. The class of the Exchangeable Occupancy Models (EOM) is actually a wide and interesting one. We then devote this Section to analyze several special aspects related with such a condition. We notice, in particular, that the most well-known occupancy models, i.e. \emph{Maxwell-Boltzmann}, \emph{Bose-Einstein} and \emph{Fermi-Dirac}, are exchangeable. They are defined as follows.

\begin{itemize}
\item Maxwell-Boltzmann:
\[
P \{ X_{1}=x_{1}, X_{2}=x_{2}, \dots, X_{n}=x_{n} \}=\frac{1}{n^{r}} \, \frac{r!}{x_{1}! x_{2}! \cdots x_{n}!},
\]
for $\mathbf{x} \equiv (x_{1}, x_{2}, \dots, x_{n}) \in A_{n,r}$. \\

\item Bose-Einstein:
\[
P \{ X_{1}=x_{1}, X_{2}=x_{2}, \dots, X_{n}=x_{n} \} = \frac{1}{\binom{n+r-1}{n-1}}
\]
for $\mathbf{x} \equiv (x_{1}, x_{2}, \dots, x_{n}) \in A_{n,r}$. \\

\item Fermi-Dirac:
\[
P \{ X_{1}=x_{1}, X_{2}=x_{2}, \dots, X_{n}=x_{n} \} = \frac{1}{\binom{n}{r}}
\]
for $\mathbf{x} \equiv (x_{1}, x_{2}, \dots, x_{n}) \in \widehat{A}_{n,r}$, where $\widehat{A}_{n,r} := \left\{ \mathbf{x} \in A_{n,r} : x_{j} \in \{ 0,1 \} \right\}$. \\
\end{itemize}

As mentioned above, we can easily see that any vector $\left(  X_{1}, \dots, X_{n} \right)$ distributed according to one of these models is an exchangeable random vector. \\

In the next Section we will introduce a wide class of EOM's that contains these fundamental models. For an account about MB, BE, FD see \citet{Fel68} and \citet{Cha05}. In the following Propositions we state some general properties of EOM's.

\begin{proposition}\label{UniMarg}
If $X_{1}, X_{2}, \dots, X_{n}$ are exchangeable, then the univariate marginals of  $Y_{1}$, $Y_{2}, \dots$, $Y_{r}$ are uniform on $\{ 1, 2, \dots,n \}$.
\end{proposition}
\begin{proof}
We prove the statement by showing that for any $a, b \in \{1, 2, \dots, n \}$, $a \neq b$, we have $P \{ Y_{1} = a \} = P \{ Y_{1} = b \}$. This would mean that $P \{ Y_{1} = a \} = \frac{1}{n}$ and hence the univariate marginal distributions of $(Y_{1}, Y_{2}, \dots, Y_{r})$ are uniform on $\{ 1, 2, \dots, n \}$. \\
Let $(y_2, \dots, y_r)$ be a $(r-1)$-tuple in $D_{r-1,n}$ and consider $a \in \{ 1, 2, \dots, n \}$. Then
\[
P\{Y_{1}=a\}=\sum_{(y_2, \dots, y_r) \in D_{r-1,n}} P \{ Y_{1}=a, Y_{2}=y_{2},  \dots, Y_{r}=y_{r} \}.
\]
By using \eqref{XvsY} we obtain that
\begin{align*}
P \{ Y_{1} = a \} &= \sum_{(y_2, \dots, y_r) \in D_{r-1,n}} \frac{P \{ X_{1} = \tilde{\varphi}_{1} (a, y_2, \dots, y_r), \dots, X_{n} = \tilde{\varphi}_{n} (a, y_2, \dots, y_r) \}}{\binom{r}{\tilde{\varphi}_{1} (a, y_2, \dots, y_r) \tilde{\varphi}_{2} (a, y_2, \dots, y_r) \cdots \tilde{\varphi}_{n} (a, y_2, \dots, y_r)}} \\
&\\
& =\sum_{(y_2, \dots, y_r) \in D_{r-1,n}} \frac{\tilde{\varphi}_{1} (a, y_2, \dots, y_r)! \cdots \tilde{\varphi}_{n} (a, y_2, \dots, y_r)!}{r!} \times \\
& \qquad \qquad \times P \{ X_{1} = \tilde{\varphi}_{1} (a, y_2, \dots, y_r), \dots, X_{n}
= \tilde{\varphi}_{n} (a, y_2, \dots, y_r)\}
\end{align*}
Taking into account that for any $j \in \{ 1, 2, \dots, n \}$, the function $\tilde{\varphi}_{j}$ is given by \eqref{tilde_varphi_function}, we obtain
\begin{align}\label{marg_a}
P \{ Y_{1} = a \} &  = \sum_{(y_2, \dots, y_r) \in D_{r-1,n}} \frac{\tilde{\varphi}_{1} (y_2, \dots, y_r)! \cdots (\tilde{\varphi}_{a} (y_2, \dots, y_r)+1)! \cdots \tilde{\varphi}_{n} (y_2, \dots, y_r)!}{r!} \times \nonumber\\
& \: \times P \{ X_{1} = \tilde{\varphi}_{1} (y_2, \dots, y_r), \dots, X_{a} = \tilde{\varphi}_{a} (y_2, \dots, y_r)+1, \dots, X_{n} = \tilde{\varphi}_{n} (y_2, \dots, y_r) \} \nonumber\\
&\nonumber\\
&  = \sum_{(y_2, \dots, y_r) \in D_{r-1,n}} \frac{\tilde{\varphi}_{1} (y_2, \dots, y_r)! \cdots \tilde{\varphi}_{n} (y_2, \dots, y_r)!}{r!} (\tilde{\varphi}_{a} (y_2, \dots, y_r)+1) \nonumber\\
& \: \times P \{ X_{1} = \tilde{\varphi}_{1} (y_2, \dots, y_r), \dots, X_{a} = \tilde{\varphi}_{a} (y_2, \dots, y_r)+1, \dots, X_{n} = \tilde{\varphi}_{n}(y_2, \dots, y_r) \}.
\end{align}
Now, we take $b \in \{ 1, 2, \dots, n \}$, $b \neq a$. Without loss of generality we can suppose that $a<b$. Then, the univariate marginal of $Y_{1}$ computed in $b$ is
\begin{align}\label{marg_b}
P \{ Y_{1} = b \} & = \sum_{(y_2, \dots, y_r) \in D_{r-1,n}} \frac{\tilde{\varphi}_{1} (y_2, \dots, y_r)! \cdots \tilde{\varphi}_{n} (y_2, \dots, y_r)!}{r!} (\tilde{\varphi}_{b} (y_2, \dots, y_r)+1) \nonumber\\
& \: \times P \{ X_{1} = \tilde{\varphi}_{1} (y_2, \dots, y_r), \dots, X_{b} = \tilde{\varphi}_{b} (y_2, \dots, y_r)+1, \dots, X_{n} = \tilde{\varphi}_{n}(y_2, \dots, y_r) \}. 
\end{align}
The sums \eqref{marg_a} and \eqref{marg_b} have $n^{r-1}$ terms and are computed using all possible $(r-1)$-tuples $(y_2, \dots, y_r) \in D_{r-1,n}$. Therefore for any $j \in \{ 1, 2,  \dots, n \}$ and $i \in \{ 0, 1, 2, \dots, r-1 \}$ there are summation terms in \eqref{marg_a} and in \eqref{marg_b}, such that $\tilde{\varphi}_{j} (y_2, \dots, y_r) = i$. Let $\alpha, \beta \in \{ 0, 1, \dots, r-1 \}$ two fixed values. Then there are summation terms in \eqref{marg_a} such that $\tilde{\varphi}_{a}(y_2, \dots, y_r) = \alpha$ and $\tilde{\varphi}_{b} (y_2, \dots, y_r) = \beta$. Also there are summation terms in \eqref{marg_b} such that $\tilde{\varphi}_{a} (y_2, \dots, y_r) = \beta$ and $\tilde{\varphi}_{b} (y_2, \dots, y_r) = \alpha$. \\
Let us now consider $(x_{1}, \dots, x_{a}, \dots, x_{b}, \dots, x_{n}) \in A_{n,r-1}$ such that $x_{a} = \alpha$ and $x_{b} = \beta$. Then, there are $\binom{r-1}{x_{1} x_{2} \cdots x_{n}}$ summation terms in \eqref{marg_a}, all equal to
\[
\frac{x_{1}! \cdots x_{n}!}{r!} (\alpha+1) P \{ X_{1} = x_{1}, \dots, X_{a}= \alpha+1, \dots, X_{b} = \beta, \dots, X_{n} = x_{n} \}
\]
The overall sum of these terms is:
\begin{equation}\label{sum_a}
\frac{\alpha+1}{r} P \{ X_{1} = x_{1}, \dots, X_{a} = \alpha+1, \dots, X_{b} = \beta, \dots, X_{n} = x_{n} \}.
\end{equation}
In the same way let us consider $(x_{1},\dots, x_{a}, \dots, x_{b}, \dots, x_{n}) \in A_{n,r-1}$ such that $x_{a} = \beta$ and $x_{b} = \alpha$. Then, there are $\binom{r-1}{x_{1} x_{2} \cdots x_{n}}$ summation terms in \eqref{marg_a}, all equal to
\[
\frac{x_{1}! \cdots x_{n}!}{r!} (\alpha+1) P \{ X_{1} = x_{1}, \dots, X_{a} = \beta, \dots, X_{b} = \alpha+1, \dots, X_{n} = x_{n} \}
\]
The overall sum of these terms is:
\begin{equation}\label{sum_b}
\frac{\alpha+1}{r} P \{ X_{1} = x_{1}, \dots, X_{a} = \beta, \dots, X_{b} = \alpha+1, \dots, X_{n} = x_{n} \}.
\end{equation}
The random vector $(X_{1}, X_{2}, \dots, X_{n})$ is exchangeable and therefore \eqref{sum_a} and \eqref{sum_b} are equal. For any $\alpha, \beta \in \{ 0, 1, \dots, r-1 \}$ the above terms uniquely determine the sums \eqref{marg_a} and \eqref{marg_b}, which then are equal. As a consequence, $P \{ Y_{1} = a \} = P \{ Y_{1} = b \}$. \\
\end{proof}

\begin{remark}
Let
\begin{equation*}
\begin{split}
(n)_r&=n(n+1)\cdots (n+r-1)\\
n^{(r)}&=n(n-1)\cdots (n-r+1)
\end{split}
\end{equation*}
be the ascending and falling factorial, respectively (see \citet{Cha05}). For MB, BE and FD, the joint distributions of the random variables $Y_1, \dots, Y_r$  are given by
\begin{itemize}
\item Maxwell-Boltzmann:
\[
P \{Y_1=y_1, \dots, Y_r = y_r\} = \frac{1}{n^r}
\]
for $\mathbf{y} \in D_{r,n}$; \\

\item Bose-Einstein:
\[
P \{Y_1=y_1, \dots, Y_r = y_r\} = \frac{\tilde{\varphi}_1(\mathbf{y})! \cdots \tilde{\varphi}_n(\mathbf{y})!}{(n+r-1)(n+r-2) \cdots (n+1) n} = \frac{\prod_{j=1}^n \tilde{\varphi}_j(\mathbf{y})!}{(n)_r}
\]
for $\mathbf{y} \in D_{r,n}$; \\

\item Fermi-Dirac:
\[
P \{Y_1=y_1, \dots, Y_r = y_r\} = \frac{\tilde{\varphi}_1(\mathbf{y})! \cdots \tilde{\varphi}_n(\mathbf{y})!}{n(n-1)(n-2) \cdots (n-r+1)} = \frac{\prod_{j=1}^n \tilde{\varphi}_j(\mathbf{y})!}{n^{(r)}}
\]
for $\mathbf{y} \in \widehat{D}_{r,n}$, where $\widehat{D}_{r,n} = \left\{\mathbf{y} \in D_{r,n} : \sum_{j=1}^r y_j \leq n \mbox{ and the $y_j$'s are all distinct}  \right\}$.
\end{itemize}
All these three distributions admit uniform marginals, in agreement with Proposition \ref{UniMarg}.
Observe that, in the MB case, the random variables $Y_1, \dots, Y_r$ are independent and their distribution is uniform over the set $\{1, 2, \dots, n\}$. \\
\end{remark}

%=============================================================================
\section{Occupancy $\mathcal{M}^{(a)}$-models and processes with the generalized UOSP}\label{Sect:M-models}
%=============================================================================

In this Section we consider a remarkable sub-class of EOM's, that includes the models MB, FD and BE.  Such a class can be introduced as follows: fix a function \mbox{$a: \{ 0, 1, \dots \} \longrightarrow ]0,+\infty[$} and for $r, n \in \mathbb{N}$, set
\begin{equation}\label{Def:EOMM}
P \{ \mathbf{X} = \mathbf{x} \} = \frac{\prod_{j=1}^n a(x_j)}{C_{n,r}^{(a)}}  \qquad \mbox{ for } \mathbf{x} \in A_{n,r} \,,
\end{equation}
where
\[
C_{n,r}^{(a)} = \sum_{\boldsymbol{\xi} \in A_{n,r}} \prod_{j=1}^n a \left( \xi_j \right) \,.
\]
In this case we say that $\mathbf{X} \equiv (X_1, \dots, X_n)$ is distributed according to the exchangeable occupancy model $\mathcal{M}_{n,r}^{(a)}$.\\

Notice that the Maxwell-Boltzmann, Bose-Einstein and Fermi-Dirac models are respectively obtained by letting
\begin{equation*}
\begin{array}{lll}
\mbox{MB: } & a(x) = \dfrac{1}{x!}, &\mbox{ for } x=0, 1, 2, \dots\\
\mbox{BE: } & a(x) = 1, &\mbox{ for }  x = 0, 1, 2, \dots\\
\mbox{FD: } & a(0) = 1, \quad a(1)=1, \quad a(x)=0, & \mbox{ for } x = 2, 3, \dots
\end{array}
\end{equation*}

\bigskip

It is interesting at this stage to point out the existence of $\mathcal{M}_{n,r}^{(a)}$-models different from the three above. An example is provided by the Pseudo-contagious models presented in \citet{Cha05}.

\begin{example}[Pseudo-contagious occupancy model]
This is the model characterized by the joint probability distribution
$$P\{X_1=x_1,\dots,X_n=x_n\} = \frac{\binom{s+x_1-1}{x_1}\cdots \binom{s+x_n-1}{x_n}}{\binom{sn+r-1}{r}} \,, \quad \mbox{ for } {\bf x}\in A_{n,r} \,.$$
This model belongs to the class $\mathcal{M}_{n,r}^{(a)}$. In fact, it corresponds to the choice of $a(x)=\binom{s+x-1}{x}$.
\end{example}

%As we can see in the next example, there is a connection between $\mathcal{M}^{(a)}$-models 
%and Gibbs Measures.
%
%\begin{example}[Gibbs measures]\label{Gibbsex}
%
%In Section 6 of \citet{BaLe07} is studied a mean-field model introduced by \citet{BEG71} (BEG). Let $\mathcal{X}=\{-1,0,1\}^N$, where $N$ is an even integer, and let $\pi$ be the probability distribution of the BEG model defined by
%$$\pi({\bf x})=Z^{-1}_N(\beta,K)\exp\{-\beta R_N({\bf x})+\frac{K\beta}{N}S_N^2({\bf x})\}\qquad K>0,\beta>0$$
%where $Z^{-1}_N(\beta,K)$ is the normalizing constant and
%$$S_N({\bf x})=\sum_{i=1}^N x_i\quad R_N({\bf x})=\sum_{i=1}^N x_i^2\quad {\bf x}=(x_1,x_2,\dots,x_N)$$
%The $\mathcal{M}_{n,r}^{(a)}$-model with $a(x)=-\beta x^2+K\beta\frac{r^2}{n^2}$, $\beta,K>0$, 
%reminds the probability distribution of the BEG model. Indeed, for ${\bf x}\in A_{n,r}$
%\begin{equation*}
%\begin{split}
%P(X_1=x_1,\dots, X_n=x_n)&=\frac{\prod_{j=1}^n a(x_j)}{C_{n,r}^{(a)}}\\
%&=\frac{\prod_{j=1}^n \exp\{-\beta x^2_j+K\beta\frac{r^2}{n^2}\}}{C_{n,r}^{(a)}}\\
%&=\frac{ \exp\{-\beta \sum_{j=1}^n x^2_j+K\beta\frac{(\sum_{j=1}^n x_j)^2}{n}\}}{C_{n,r}^{(a)}}\\
%\end{split}
%\end{equation*}
%\end{example}

\medskip 

In a $\mathcal{M}_{n,r}^{(a)}$-model, the joint distribution of the $Y_1,\dots, Y_r$ random variables introduced in Section~\ref{Sect:EOMandUOSP} is
\begin{equation*}\label{jointYsEOMM}
P \{ Y_{1} = y_{1}, Y_{2} = y_{2}, \dots, Y_{r} = y_{r} \} = \frac{\prod _{l=1}^{n} a(\tilde{\varphi}_{l} (\mathbf{y})) \tilde{\varphi}_{l} (\mathbf{y})!}{r!C_{n,r}^{(a)}}.
\end{equation*}

\medskip

An approach for constructing $\mathcal{M}_{n,r}^{(a)}$-models is described in \citeauthor{Cha05} \citep[Chapter 4]{Cha05} and makes use of i.i.d. random variables. It works as follows. Let $Z_1,\dots,Z_n$ be $\mathbb{N}$-valued i.i.d. random variables with common law
$$P\{Z_i=x\}=q(x)\qquad \mbox{ for } i=1,\dots,n.$$
and let $S_n=Z_1+\cdots+Z_n$. \\
Consider the random vector $(X_1,\dots,X_n)$ with values in $A_{n,r}$ and such that
$$P\{X_1=x_1,\dots,X_n=x_n\} := P\{Z_1=x_1,\dots,Z_n=x_n|S_n=r\}.$$

This means that, for any $(x_1,\dots,x_n) \in A_{n,r}$, it holds
\begin{align}\label{qEOM}
P\{X_1=x_1,\dots,X_n=x_n\} &= \frac{P\{Z_1=x_1,\dots,Z_n=x_n\}}{P\{S_n=r\}} %\nonumber \\ %\qquad {\bf x}\in A_{n,r}
%& \nonumber \\
%&
= \frac{\prod_{j=1}^n q(x_j)}{P\{S_n=r\}} \,. %\qquad {\bf x}\in A_{n,r}
\end{align}

The Formula \eqref{Def:EOMM} can be seen as a slight generalization of Formula \eqref{qEOM}; actually $a(x)$ is not necessarily a probability distribution. We can furthermore extend as follows the construction of the approach based on i.i.d. variables. Let $Z_{1}, \dots, Z_{n}$ be $\mathbb{N}$-valued conditionally i.i.d. random variables with joint discrete density
\begin{equation}\label{ExponModel}
P \{ Z_{1} = z_{1}, \dots, Z_{n} = z_{n} \} = \prod_{j=1}^{n} a \left( z_{j} \right) \int_{0}^{+\infty} \exp \left\{ -\theta \sum_{j=1}^{n} z_{j} \right\} \Lambda \left( d\theta \right),
\end{equation}
$\Lambda$ being a probability distribution on the positive real half-line. Put $S_{n} = \sum_{j=1}^{n} Z_{j}$ and let us repeat the same construction as before: consider the random vector $\left( X_{1}, \dots, X_{n} \right)$ with values in $A_{n,r}$ and such that, for $\mathbf{x} \in A_{n,r}$,
\[
P \{ X_{1} = x_{1}, \dots, X_{n} = x_{n} \} := P \{ Z_{1} = x_{1}, \dots, Z_{n} = x_{n} | S_{n} = r \}.
\]

Then, as it is easily seen, we reobtain the occupancy model in \eqref{Def:EOMM}. Indeed, we have
\begin{align*}
P \{ Z_{1} = x_{1}, \dots, Z_{n} = x_{n} | S_{n} = r \} &= \frac{P \{ Z_{1} = x_{1}, \dots, Z_{n} = x_{n} \}}{P \{ S_{n} = r \}} \\
&\\
&= \frac{\prod_{j=1}^{n} a \left( x_{j} \right) \int_{0}^{+\infty} \exp \{ -\theta r \} \Lambda \left( d\theta \right)}{\sum_{\mathbf{v}\in A_{n,r}} \prod_{j=1}^{n} a \left( v_{j} \right) \int_{0}^{+\infty} \exp \{ -\theta r \} \Lambda \left( d\theta \right)} \\
&\\
&=\frac{\prod_{j=1}^{n} a \left( x_{j} \right)}{C_{n,r}^{(a)}}.
\end{align*}

\begin{remark}\label{Rmk:sufficiency}
The fact that $P \{ Z_{1} = x_{1}, \dots, Z_{n} = x_{n} | S_{n} = r \}$ does not depend on the distribution $\Lambda$, has an immediate interpretation in statistical terms: for the \emph{exponential model} in \eqref{ExponModel}, $S_{n}$ is a \emph{sufficient statistic} with respect to the \emph{parameter} $\theta$. 
\end{remark}

The construction considered above is somehow more general than the one based on i.i.d. variables. Anyway, the Formula \eqref{qEOM} suggests some important remarks.

\begin{remark}
Up to a multiplicative factor, the normalization constant $C_{n,r}^{(a)}$ in \eqref{Def:EOMM} can be interpreted as the probability that a sum  of certain i.i.d. random variables is equal to $r$, i.e. $C_{n,r}^{(a)} = k P\{S_n=r\}$ with $k$ suitable constant.
\end{remark}

\begin{remark}
With the above construction the three fundamental models can be recovered by choosing suitably the distribution $q(x)$. In fact, MB, BE and FD correspond to the cases when $q(x)$ obeys a Poisson, Geometric and Bernoulli distribution, respectively. Moreover, the pseudo-contagious occupancy model is obtained by setting $q(x)$ a negative Binomial distribution.\\
\end{remark}
%We now conclude this section by observing some specific aspects concerning the models of the type $\mathcal{M}_{n,r}^{(a)}$ defined in \eqref{Def:EOMM}. \\

In the remaining part of this Section we want to generalize Theorem \ref{Thm:UOSP's}. More precisely, we 
will define the class of the processes with the generalized UOSP and provide a related characterization  in terms of $\mathcal{M}^{(a)}$-models.\\

%come back to the interpretation of a conditional counting process as an occupancy $\mathcal{M}^{(a)}$-model to derive the generalized UOSP and to provide a characterization of processes obeying the latter.\\

%=============================================================================
%\section{Processes with the generalized UOSP}\label{Sect:genUOSP}
%=============================================================================

For $M \in \mathbb{N} \cup \{+\infty\}$, let $\{ N_{t} \}_{t=0, 1, \dots,M}$ be a discrete-time counting process with \emph{jump amounts} $J_{0}, J_{1}, \dots, J_{M}$, i.e. $J_{0}, J_{1}, \dots, J_{M}$ is a sequence of $\{ 0, 1, 2, \dots \}$-valued random variables such that, for $t = 0, 1, \dots, M$,
\begin{equation}\label{Def:CountProc}
N_{t} = \sum_{h=0}^{t} J_{h}.
\end{equation}

For our purposes we introduce the following notation and definitions.

Let $T_{1}, T_{2}, \dots$ denote the \emph{arrival times} of the process $\{ N_{t} \}_{t = 0, 1, \dots}$, i.e.
\begin{equation}\label{Def:ArrTimes}
T_{n} = \inf \{ t \geq 0 : N_{t} \geq n \},
\end{equation}
and let $Z_{1}, Z_{2}, \dots$ denote the \emph{inter-arrival times}, i.e.
\[
Z_{n} =T_{n} -T_{n-1} \,.
\]
Notice that, since $P \{ J_{h} > 1 \} > 0$, it can happen $\{ T_{n-1} = T_{n} \}$, and then $\{ Z_{n} = 0 \}$, for some $n$. The zeros of the $Z$'s are related to the \emph{ties} of the $T$'s and to the jump amounts greater than one.

% The relations existing among the joint probability distributions of the variables J's, Z's and T's will be analyzed along the proof of Theorem 5.1, below.

%The jump amounts could be defined through the arrival time, i.e.
%$$J_h=\sum_i {\bf 1}_{\lbrace T_i=h\rbrace}$$
%It is also possible to connect the inter-arrival times $Z_1,Z_2,\dots$ with the jump amounts $J_0,J_1,\dots$, but we need to define two auxiliary sequences $\lbrace I_n\rbrace_{n=1,2,\dots}$ and $\lbrace \tilde{J}_n\rbrace_{n=1,2,\dots}$. Let $I_1,I_2,\dots$ be the sequence of indexes such that the $J$'s are strictly positive, i.e.
%\begin{equation*}
%\begin{split}
%I_1&=\inf\lbrace i| J_i>0\rbrace\\
%I_n&=\inf\lbrace i| i>I_{n-1}\mbox{ and }J_i>0\rbrace\qquad n>1
%\end{split}
%\end{equation*}
%and
%$$\tilde{J}_n=\sum_{l=1}^{n} J_{I_l}\qquad n=1,2,\dots$$
%Note that
%
%Then
%\begin{equation}
%\begin{array}{lllll}\label{eq:JZ}
%Z_1&=I_1\quad &Z_2&=0\cdots Z_{\tilde{J}_1}&=0\\
%Z_{\tilde{J}_1+1}&=I_2-I_1\quad &Z_{\tilde{J}_1+2}&=0\cdots Z_{\tilde{J}_2}&=0\\
%\vdots &\vdots &\vdots &\vdots &\vdots\\
%Z_{\tilde{J}_{n-1}+1}&=I_n-I_{n-1}\quad &Z_{\tilde{J}_{n-1}+2}&=0\cdots Z_{\tilde{J}_n}&=0\\
%\vdots &\vdots &\vdots &\vdots &\vdots \\
%\end{array}
%\end{equation}
%The equivalence above will be useful in the sequel.
\begin{definition}
Let $a:\{0,1,\dots\} \longrightarrow ]0,+\infty[$ be a given function and $\{ N_{t} \}_{t = 0, 1, \dots}$ a discrete-time counting process with jump amounts $J_0, \dots, J_t$. We say that it satisfies the $\mathcal{M}^{(a)}$-Uniform Order Statistics Property ($\mathcal{M}^{(a)}$-UOSP) if, for any $t, k \in \mathbb{N}$ and any $(j_{0}, \dots, j_{t}) \in A_{t+1,k}$, we have
\begin{equation}\label{Def:M-UOSP}
P \{ J_{0} = j_{0}, \dots, J_{t} = j_{t} | N_{t} = k \} = \frac{\prod_{h=0}^t a(j_{h})}{C_{t+1,k}^{(a)}}.
\end{equation}
\end{definition}
This definition can be seen as a natural, and unifying extension of the definition of discrete UOSP given in \citet{ShSpSu04,ShSpSu08}. In fact, by choosing function $a$ as in FD, MB and BE models, the  UOSP$(<)$ and UOSP$(\leq_{1,2})$ are respectively recovered.\\
The following definition, in its turn, can be seen as an appropriate generalization of the definition of mixed geometric process  given in \citeauthor{HuSh94} \citep[Section 4]{HuSh94}.

\begin{definition}\label{Def:a-mixgeom}
Let $a: \{ 0, 1, \dots \} \longrightarrow ]0,+\infty[$ be a given function such that $a(0)=1$. The process $\{ N_{t} \}_{t = 0, 1, \dots, M}$ is an $a$\emph{-mixed geometric process} if the discrete joint density of $(J_{0}, J_{1}, \dots, J_{t} )$ has the form
\begin{equation}\label{CondDefMixedaProc}
p_{t} (j_{0}, j_{1}, \dots, j_{t}) = R_{t} \left( \sum_{h=0}^{t} j_{h} \right) \cdot \prod_{h=0}^{t} a(j_{h}) \quad \mbox{for } t = 0, 1, \dots, M
\end{equation}
for a suitable sequence of functions $R_{t}: \{0, 1, \dots \} \longrightarrow \mathbb{R}_{+}$, $t=0,1,\dots,M$.
\end{definition}
Notice that $M$ in the above definition can be a natural number or $+\infty$.
\begin{remark}
The sequence of functions $R_{1}, R_{2}, \dots$ in the Definition \ref{Def:a-mixgeom} cannot be independent of function $a$. In fact, since the discrete density $p_{t-1}$ must be the marginal of $p_{t}$, from \eqref{CondDefMixedaProc} we obtain
\[
R_{t-1}(k) = \sum_{l=0}^{+\infty} a(l) R_{t} (k+l).
\] 
It is easy to see that $R_{t}(k)$ must be related to the coefficients $C_{t,k}^{(a)}$ (in this respect, see also Formula \eqref{R-eq} below).
\end{remark}

\begin{remark}
Similarly to what had been developed in \citet{HuSh94} we notice that an $a$-mixed geometric process is a (non-homogeneous) Markov chain. More precisely, a direct use of the condition \eqref{CondDefMixedaProc} readily yields:
\[
P \{ N_{t+1} = k+i | N_{t} = k \} = a(i) \, \frac{R_{t+1}(k+i)}{R_t(k)} \,.
%= a(i) \, \frac{C_{t+1,k}^{(a)}}{C_{t+2,k+i}^{(a)}} \, \frac{P\{N_{t+1} =k+i\}}{P\{N_{t} =k\}} \,.
\]
This property is related with the sufficiency property that was pointed out in Remark~\ref{Rmk:sufficiency}.
\end{remark}

%\begin{remark}
%An occupancy model whose occupancy numbers do not satisfy a mixed geometric property of the type \eqref{CondDefMixedaProc} cannot belong to the class of $\mathcal{M}^{(a)}$-models, for any function $a$.
%\end{remark}

The following characterization theorem extends the result in \citeauthor{ShSpSu04} \citep[Section 4]{ShSpSu04}. It is based on very simple relations that, in the case of $a$-mixed geometric processes, tie the joint probability distributions of the variables $J$'s and those of $Z$'s and $T$'s.

\begin{theorem}\label{THM}
Let $M\in\mathbb{N}$ or $M=+\infty$ and let $\{ N_{t} \}_{t = 0, 1, \dots, M}$ be a discrete-time counting process. Then, the following statements are equivalent:
\begin{enumerate}
\item $\{ N_{t} \}_{t = 0, 1, \dots, M}$ satisfies the $\mathcal{M}^{(a)}$-UOSP.

\item $\{ N_{t} \}_{t = 0, 1, \dots, M}$ is an $a$-mixed geometric process.

\item A sequence of functions $R_n:\mathbb{N}\longrightarrow \mathbb{R}_{+}$ exists such that, for any $k \in \mathbb{N}$, 
\begin{equation}\label{Thm:3}
P\lbrace Z_1=z_1,\dots,Z_k=z_k, Z_{k+1}>0\rbrace=R_{\sum_{i=1}^{k} z_{i}}(k)\cdot \prod_{h=0}^{\sum_{i=1}^{k} z_{i}} a \left( \sum_{i=1}^k \mathbf{1}_{ \left\{ \sum_{d=1}^{i} z_{d}=h \right\} } \right) \,.
\end{equation}

\item A sequence of functions $R_n:\mathbb{N}\longrightarrow \mathbb{R}_{+}$ exists such that, for any $\chi \in \mathbb{N}$, 
\begin{equation}\label{Thm:4}
P \{ T_{1} = t_{1}, \dots, T_{\chi} = t_{\chi}, T_{\chi+1}>t_{\chi} \} = R_{t_{\chi}} (\chi) \cdot \prod_{h=0}^{t_{\chi}} a \left( \sum_{i=1}^{\chi} \mathbf{1}_{ \{ t_{i}=h \}} \right).
\end{equation}
\end{enumerate}
\end{theorem}

\begin{proof}
(i) $\Rightarrow$ (ii). Consider $\{N_t\}_{t=0,1 \dots, M}$ satisfying the $\mathcal{M}^{(a)}$-UOSP. Clearly \eqref{CondDefMixedaProc} is satisfied. In fact we can write
\begin{align*}
P \{ J_{0} = j_{0}, \dots, J_{t} = j_{t}\} &= P \left\{ J_{0} = j_{0}, \dots, J_{t-1} = j_{t-1}, J_t=j_t, N_{t} = \sum_{h=0}^{t} j_{h} \right\} \\ %&\\
&= P \left\{ J_{0} = j_{0}, \dots, J_{t} = j_{t} \left\vert N_{t} = \sum_{h=0}^{t} j_{h} \right. \right\} \, P \left\{ N_{t} = \sum_{h=0}^{t} j_{h} \right\} \\
%&\\
&= \frac{\prod_{h=0}^t a(j_{h})}{C_{t+1,\sum_{h=0}^{t} j_{h}}^{(a)}} \, P \left\{ N_{t} = \sum_{h=0}^{t} j_{h} \right\} \,,
\end{align*}
where the last equality is due to \eqref{Def:M-UOSP}. Thus \eqref{CondDefMixedaProc} is readily obtained by setting, for any $k\in\mathbb{N}$

\begin{equation}\label{R-eq}
R_{t}(k) := \frac{P\{N_{t}=k\}}{C_{t+1,k}^{(a)}} \,.
\end{equation}

(ii) $\Rightarrow$ (i). Consider an $a$-mixed geometric process $\{N_t\}_{t=0,1 \dots,M}$. For any $t, k \in \mathbb{N}$ and for any $\left( j_{0}, \dots, j_{t} \right)  \in A_{t+1,k}$, by \eqref{CondDefMixedaProc} we get
\begin{align*}
P \{ J_{0} = j_{0}, \dots, J_{t} = j_{t} \vert N_{t} = k \} &= \frac{P \{ J_{0} = j_{0}, \dots, J_{t} = j_{t} \}}{P \{ N_{t} = k \}}\\
%&\\
&= \frac{R_{t}(k)\cdot \prod_{h=0}^{t} a(j_{h})}{\sum_{\mathbf{v} \in A_{t+1,k}} R_{t}(k)\cdot \prod_{h=0}^{t} a(v_{h})} =\frac{\prod_{h=0}^t a(j_{h})}{C_{t+1,k}^{(a)}}\,,
\end{align*}
which is exactly \eqref{Def:M-UOSP}. \\

(iii) $\Rightarrow$ (iv). For any $\chi \in \mathbb{N}$, consider $t_1, \dots, t_{\chi}$ positive integers. By taking into account Eq. \eqref{Thm:3}, we can write
\begin{align*}
P \{T_1 =t_1, \dots, T_{\chi} = t_{\chi}, T_{\chi+1}>t_{\chi} \} &= P \{Z_1 = t_1, Z_2 = t_2-t_1, \dots, Z_{\chi} = t_{\chi} - t_{\chi-1}, Z_{\chi+1}>0 \} \\
&= R_{\sum_{d=1}^{\chi} (t_d - t_{d-1})} (\chi) \! \prod_{h=0}^{\sum_{d=1}^{\chi} (t_d - t_{d-1})} \! \! a \left( \sum_{i=1}^{\chi} \mathbf{1}_{\left\{ \sum_{d=1}^i (t_d - t_{d-1}) = h \right\}}\right) \\
&= R_{t_{\chi}} (\chi) \cdot \prod_{h=0}^{t_{\chi}} a \left( \sum_{i=1}^{\chi} \mathbf{1}_{\{ t_i = h \}}\right) \,.
\end{align*} 
\\
(iv) $\Rightarrow$ (iii). For any $k \in \mathbb{N}$, consider $z_1, \dots, z_k$ positive integers. We have
%\begin{align*}
\[
P \{ Z_{1} = z_{1}, \dots, Z_{k} = z_{k}, Z_{k+1}>0 \} = P \left\{ T_{1} = z_{1}, T_{2} = z_{1} + z_{2}, \dots, T_{k} = \sum_{d=1}^{k} z_{d}, T_{k+1}>T_k \right\}
\]
%&\\
%&= R_{\sum_{i=1}^{k}z_{i}}(k)\cdot \prod_{h=0}^{t} a \left( \sum_{i} \mathbf{1}_{ \{ \sum_{d=1}^{i} z_{d}=h \}} \right)
%\end{align*}
and the conclusion immediately follows by \eqref{Thm:4}. \\

(ii) $\Rightarrow$ (iv). Notice that, by definition of the variables $J_{0},J_{1},\dots$ and since
$T_{1}\leq T_{2}\leq\dots$, an event of the form $\{T_{1}=t_{1},\dots,T_{\chi
}=t_{\chi},T_{\chi+1}>t_{\chi}\}$ is equivalent to
\[
\left\{ J_{0}=\sum_{i=1}^{\chi}\mathbf{1}_{\{t_{i}=0\}},J_{1}=\sum_{i=1}^{\chi
}\mathbf{1}_{\{t_{i}=1\}},\dots,J_{t_{\chi}}=\sum_{i=1}^{\chi}\mathbf{1}%
_{\{t_{i}=t_{\chi}\}} \right\},
\]
for any $\chi\in\mathbb{N}$ and any $0\leq t_{1}\leq\dots\leq t_{\chi}$. We can
then write 
\begin{multline*}
P\{T_{1}=t_{1},\dots,T_{\chi}=t_{\chi},T_{\chi+1}>t_{\chi}\}=\\
=P \left\{ J_{0}=\sum_{i=1}^{\chi}\mathbf{1}_{\{t_{i}=0\}},J_{1}=\sum_{i=1}^{\chi
}\mathbf{1}_{\{t_{i}=1\}},\dots,J_{t_{\chi}}=\sum_{i=1}^{\chi}\mathbf{1}%
_{\{t_{i}=t_{\chi}\}} \right\}.
\end{multline*}

Thus, by assuming that $\{N_{t}\}_{t=0,1,\dots,M}$ is an $a$-mixed
geometric process, Eq. \eqref{CondDefMixedaProc} yields
\[
P \{ T_{1}=t_{1},\dots,T_{\chi}=t_{\chi},T_{\chi+1}>t_{\chi}\}=R_{t_{\chi}}(\chi) \cdot
\prod_{h=0}^{t_{\chi}}
a \left( \sum_{i=1}^{\chi} \mathbf{1}_{\{t_{i}=h\}} \right).
\]

(iv) $\Rightarrow$ (ii). Similarly to what noticed in the previous step, we have by definition that the
event $\{J_{0}=j_{0},J_{1}=j_{1},\dots,J_{m}=j_{m}\}$ is equivalent to one of
the form
\[
\{T_{1}=t_{1},\dots,T_{r}=t_{r},T_{r+1}=\dots=T_{\chi}=m,T_{\chi+1}>m\}
\]
for $0 \leq r < \chi$ such that $j_m = \chi-r+1$ and $0 \leq t_1 \leq t_2 \leq \cdots \leq t_r < m$ such that
\[
\left\{ j_{0}=\sum_{i=1}^{\chi}\mathbf{1}_{\{t_{i}=0\}},j_{1}=\sum_{i=1}^{\chi
}\mathbf{1}_{\{t_{i}=1\}},\dots,j_{m}=\sum_{i=1}^{\chi}\mathbf{1}_{\{t_{i}%
=m\}} \right\}.
\]

Whence
\begin{equation}\label{NuovaEquazFabio}
P\{J_{0}=j_{0},J_{1}=j_{1},\dots,J_{m}=j_{m}\} = P\{T_{1}=t_{1},\dots,T_{r}=t_{r},T_{r+1}= \cdots =m,T_{\chi+1}>m\}.
\end{equation}

In the case when the condition (iv) holds, the computation of the probability
appearing in the r.h.s. only requires the knowledge of $\chi$, $t_{\chi}=m$,
and
\[
\sum_{i=1}^{\chi}\mathbf{1}_{\{t_{i}=0\}}=j_{0},\dots,\sum_{i=1}^{\chi
}\mathbf{1}_{\{t_{i}=m\}}=j_{m}\text{,}%
\]
as the formula \eqref{Thm:4} shows. Thus, by combining the formula \eqref{Thm:4} with
(\ref{NuovaEquazFabio}), we obtain \eqref{CondDefMixedaProc}.
\end{proof}

\begin{remark}
Notice that if one of the equivalent conditions (i)--(iv) of Theorem \ref{THM} holds, then in particular we get
\[
P \{ N_t=0 \} = P \{ T_1 > t \} = P \{ J_0 = 0, \dots, J_t = 0 \} = R_t(0) \,.
\]
\end{remark}

Here, we conclude this section with a few comments about the processes with the
$\mathcal{M}^{\left(  a\right)  }$-UOSP. As it is clear from the definition,
the role of such a property is played around conditioning w.r.t. events of the
form $\{N_{t}=k\}$.

Let us  consider, in the beginning, a general sequence of exchangeable jump
amounts $J_{1},J_{2},\dots$ and put $N_{s}=\sum_{i=1}^{s}J_{i}$,
$s=1,2, \dots$. Then, conditionally on $\{N_{t}=k\}$, the joint distribution of
$J_{1},J_{2}, \dots, J_{t}$ is an EOM over $A_{t,k}$. Let
$Y_{1}^{\left(  t\right)  }, \dots, Y_{k}^{\left(  t\right)  }$ be the
$\{1,2, \dots,t\}$- valued, exchangeable, random variables that correspond to
such an occupancy model. We denote by $\mathcal{L}^{\left(  t,k\right)  }$ their joint probability law and recall that their common marginal distribution is, in any case, uniform over $\{1,2, \dots,t\}$.  

In the special case of processes $\{N_{s}\}_{s=1,2,\dots}$ with the
$\mathcal{M}^{\left(  a\right)  }$-UOSP, $\mathcal{L}^{\left(  t,k\right)  }$
is given by
\[
P \left\{ Y^{(t)}_{1} = y_{1}, Y^{(t)}_{2} = y_{2}, \dots, Y^{(t)}_{k} = y_{k} \right\} = \frac{\prod_{h=0}^t a(\tilde{\varphi}_{h} (\mathbf{y})) \tilde{\varphi}_{h} (\mathbf{y})!}{k! \, C_{t+1,k}^{(a)}} \,.
\]
When more in particular $a(x)=1/x!$, i.e. when the conditional distribution of
$J_{1},J_{2}, \dots, J_{t}$ is the MB model over $A_{t,k}$, $Y_{1}^{\left(
t\right)  }, \dots, Y_{k}^{\left(  t\right)  }$ are i.i.d. with uniform
distribution over $\{1,2, \dots, t\}$. This case, the one denoted by UOSP($\leq_1$) in \citet{ShSpSu04}, can be seen in a sense as the discrete-time analog of
the homogeneous Poisson process (in continuous time). The latter satisfies in
fact the classical \emph{Order Statistics Property}: conditionally on $\{N_{t}=k\}$, the
arrival times $T_{1},\dots,T_{k}$ can be seen as the order statistics of $k$
independent random variables,  with an uniform distribution over the interval
$\left[  0,t\right]$. The set of the arrivals can thus be seen as a random
sample from a uniform distribution over $\left[  0,t\right] $.
In the case of UOSP($\leq_1$) then we have that, for any $k>1$ and $t>1$,
$\mathcal{L}^{\left(  t,k-1\right)  }$ obviously coincides with the
$(k-1)$-dimensional marginal of $\mathcal{L}^{\left(  t,k\right)  }$. This last
circumstance actually is a weaker condition than UOSP($\leq_1$). We shall see in the next section, in fact, that it holds for any $\mathcal{M}^{\left(  a\right)  }$-UOSP process with $a$ satisfying an appropriate condition, namely Eq. \eqref{condEOM} below. It can however be still considered as a condition of randomness for the arrivals. An interpretation in this direction can be in fact derived from some arguments in the next section, where we will consider the operation of dropping one arrival time (see the transformation $\mathcal{K}_{1}$ and Proposition \ref{Prop:elimination_y}).

\section{Closure under some transformations of occupancy models}\label{Sect:Transformations}
%=============================================================================

We devote this final Section to investigate further and collateral properties of occupancy models. In particular, we now consider some special transformations, that map one occupancy model into another, preserving structure and main features.  The transformations we introduce can be described as follows.\\
\begin{description}
\item[Transformation $\mathcal{K}_{1}$] --- Consider $r$ particles distributed among $n$ cells according to a given occupancy model and let $X_{1}, X_2, \dots, X_{n}$ be the related occupancy numbers. We drop one of the particles from this population  randomly  (i.e. so that each of the $r$ particles has the same probability to be dropped, independently of its cell). We denote by $X_{1}^{\prime}, X_2^{\prime}, \dots, X_{n}^{\prime}$ the occupancy numbers associated with the new population of $r-1$ particles.\\
This transformation will be denoted by $\mathcal{K}_{1}$. Notice that $\mathcal{K}_{1}: \mathcal{P} \left( A_{n,r} \right)  \longrightarrow \mathcal{P} \left( A_{n,r-1} \right)$.\\

\item[Transformation $\mathcal{K}_{2}$] --- We consider the transformation $\mathcal{K}_{2}: \mathcal{P} \left( A_{n,r} \right) \longrightarrow \mathcal{P} \left( A_{n-1,r} \right)$ simply obtained by ``erasing'' one of the $n$ cells. More precisely, we start once again from an occupancy model in $\mathcal{P} \left( A_{n,r} \right)$ and let $X_{1}, \dots, X_{n}$ denote occupancy numbers jointly distributed according to such a model. We now consider the case where the $n$-th cell is eliminated and any of the $X_{n}$ particles that had fallen within it, is put at random, and independently, within the remaining cells.\\
By denoting $X_{1}^{\prime\prime}, \dots, X_{n-1}^{\prime\prime}$ the occupancy numbers obtained by this procedure, and for $\left( x_{1}^{\prime\prime}, \dots, x_{n-1}^{\prime\prime} \right) \in A_{n-1,r}$, we can write%
\begin{multline}\label{ProbTransfK2}
P \{ X_{1}^{\prime\prime} = x_{1}^{\prime\prime}, \dots, X_{n-1}^{\prime\prime} = x_{n-1}^{\prime\prime} \} =\\ 
= \sum_{x=0}^{r} \: \sum_{\boldsymbol{\xi} \in A_{n-1,x}} \frac{\binom{x}{\xi_{1} \cdots \xi_{n-1}}}{\left( n-1 \right)^{x}} P \{ X_{1} = x_{1}^{\prime\prime} - \xi_{1}, \dots, X_{n-1} = x_{n-1}^{\prime\prime} - \xi_{n-1}, X_{n} = x \}, 
\end{multline}
where we obviously mean $P \{ X_{1} = x_{1}, \dots, X_{n} = x_{n} \} = 0$, whenever some of the coordinates $x_{1}, \dots, x_{n-1}$ turns out to be smaller than zero.\\

\item[Transformation $\mathcal{K}_{n,s}^{(N,r)}$] --- Here and in the rest of the section we will use the notation $S_{N} = \sum_{i=1}^{N} X_{i}$. We consider the transformation that maps $\mathcal{P} \left( A_{N,r} \right)$ onto $\mathcal{P} \left( A_{n,s} \right)$ (with $1 \leq n \leq N-1, 1 \leq s \leq r$) that is simply obtained by conditioning on a fixed value $s$ for the partial sum $S_{n}$. For variables $X_{1}, \dots, X_{N}$ distributed according to a given occupancy model in $\mathcal{P} \left( A_{N,r} \right)$ we consider the model in $A_{n,s}$ defined by
\[
P \{ X_{1} = x_{1}, \dots, X_{n} = x_{n} | S_{n} = s \} = \frac{P \{ X_{1} = x_{1}, \dots, X_{n} = x_{n} \}}{P \{ S_{n} = s \}}
\]
for $\left( x_{1}, \dots, x_{n} \right) \in A_{n,s}$ and with
\begin{multline*}
P \{ X_{1} = x_{1}, \dots, X_{n} = x_{n} \} =\\
= \sum_{\boldsymbol{\eta} \in A_{N-n,r-s}} P \{ X_{1} = x_{1}, \dots, X_{n} = x_{n}, X_{n+1} = \eta_{1}, \dots, X_{N} = \eta_{N-n} \},
\end{multline*}
\[
P \{ S_{n} = s\} = \sum_{\mathbf{x} \in A_{n,s}} \sum_{\boldsymbol{\eta} \in A_{N-n,r-s}} P \{ X_{1} = x_{1}, \dots, X_{n} = x_{n}, X_{n+1} = \eta_{1}, \dots, X_{n} = \eta_{N-n} \}.
\]
This transformation will be denoted by $\mathcal{K}_{n,s}^{(N,r)}$.\\
\end{description}

We start presenting an interpretation of the marginal distributions of the variables $Y_{1}, Y_{2}, \dots,Y_{r}$, introduced in Section \ref{Sect:EOMandUOSP}, in terms of the associated occupancy numbers $X_{1}, X_{2}, \dots, X_{n}$.

\begin{lemma}\label{Lmm:elimination}
Let $(X_1, \dots, X_n)$ be an occupancy model on $A_{n,r}$ and $(X_1^\prime, \dots, X_n^\prime)$ be the occupancy numbers obtained by applying $\mathcal{K}_1$. Consider the event $E_{\mathbf{x}^{\prime}} = \left\{ X_{1}^{\prime} = x_{1}^{\prime}, \dots, X_{n}^{\prime} = x_{n}^{\prime} \right\}$. Then,
\begin{equation}\label{X'}
P \left\{ E_{\mathbf{x}^{\prime}} \right\} = \sum_{h=1}^{n} \frac{x_{h}^{\prime}+1}{r} P \{ X_{1} = x_{1}^{\prime}, \dots, X_{h} = x_{h}^{\prime}+1, \dots, X_{n} = x_{n}^{\prime} \}.
\end{equation}
\end{lemma}

\begin{proof}
The proof consists of simple manipulations. In fact, we can write
\begin{align*}
P \left\{ E_{\mathbf{x}^{\prime}} \right\} & = P \{ X_{1}^{\prime}=x_{1}^{\prime}, \dots, X_{n}^{\prime}=x_{n}^{\prime} \} \\
%&\\
& = \sum_{h=1}^{n} P \{ E_{\mathbf{x}^{\prime}} \cap \left( \mbox{the eliminated particle is from the cell } h \right) \}\\
%&\\
& = \sum_{h=1}^{n} P \{E_{\mathbf{x}^{\prime}} \cap \left( X_{1} = x_{1}^{\prime}, \dots, X_{h} = x_{h}^{\prime}+1, \dots, X_{n} = x_{n}^{\prime} \right) \}\\
%&\\
&  = \sum_{h=1}^{n} P \{E_{\mathbf{x}^{\prime}} \vert X_{1} = x_{1}^{\prime}, \dots, X_{h} = x_{h}^{\prime}+1, \dots, X_{n}= x_{n}^{\prime} \} \times \\
& \qquad \qquad \quad \times P \{ X_{1} = x_{1}^{\prime}, \dots, X_{h} = x_{h}^{\prime}+1, \dots, X_{n} = x_{n}^{\prime} \} 
%&\\
%& = \sum_{h=1}^{n} \frac{x_{h}^{\prime}+1}{r} P \{ X_{1} = x_{1}^{\prime}, \dots, X_{h} = x_{h}^{\prime}+1, \dots, X_{n} = x_{n}^{\prime} \} \,.
\end{align*}
whence \eqref{X'} is readily obtained.
\end{proof}

Let us now consider the exchangeable vectors $\left( Y_{1}, \dots, Y_{r} \right)$ and $\left( Y_{1}^{\prime}, \dots, Y_{r-1}^{\prime} \right)$, corresponding to the occupancy numbers $X_{1}, \dots, X_{n}$ and $X_{1}^{\prime}, \dots, X_{n}^{\prime}$,  respectively.

\begin{proposition}\label{Prop:elimination_y}
The joint distribution of $\left( Y_{1}^{\prime}, \dots, Y_{r-1}^{\prime} \right)$ coincides with the $(r-1)$-di\-men\-sion\-al marginal distribution of $\left( Y_{1}, \dots, Y_{r} \right)$.
\end{proposition}

\begin{proof}
In view of \eqref{XvsY} and \eqref{X'}, the joint distribution of $\left( Y_{1}^{\prime}, \dots, Y_{r-1}^{\prime} \right)$ is given by
\begin{align*}
P \{ Y_{1}^{\prime} = y_{1}^{\prime}, & Y_{2}^{\prime} = y_{2}^{\prime}, \dots, Y_{r-1}^{\prime} = y_{r-1}^{\prime} \} =\\
&\\
&= \frac{P \{ X_{1}^{\prime} = \tilde{\varphi}_{1} (\mathbf{y}^{\prime}), X_{2}^{\prime} = \tilde{\varphi}_{2} (\mathbf{y}^{\prime}), \dots, X_{n}^{\prime} = \tilde{\varphi}_{n} (\mathbf{y}^{\prime}) \}}{\binom{r-1}{\tilde{\varphi}_{1} (\mathbf{y}^{\prime}) \tilde{\varphi}_{2} (\mathbf{y}^{\prime}) \cdots \tilde{\varphi}_{n} (\mathbf{y}^{\prime})}} \\
&\\
&=\frac{\tilde{\varphi}_{1} (\mathbf{y}^{\prime})! \tilde{\varphi}_{2} (\mathbf{y}^{\prime})! \cdots \tilde{\varphi}_{n}(\mathbf{y}^{\prime})!}{(r-1)!} \sum_{h=1}^{n} \frac{\tilde{\varphi}_{h} (\mathbf{y}^{\prime})+1}{r} \times \\
& \qquad \times P \{ X_{1} = \tilde{\varphi}_{1}  (\mathbf{y}^{\prime}), \dots, X_{h} = \tilde{\varphi}_{h} (\mathbf{y}^{\prime})+1, \dots, X_{n} = \tilde{\varphi}_{n}(\mathbf{y}^{\prime}) \}.
\end{align*}
For brevity's sake, let us set
\[
\boldsymbol{\tilde{\varphi}}^{(h)}(\mathbf{y}^{\prime}) := \left( \tilde{\varphi}_{1} (\mathbf{y}^{\prime}), \dots, \tilde{\varphi}_{h}(\mathbf{y}^{\prime})+1, \dots,\tilde{\varphi}_{n} (\mathbf{y}^{\prime}) \right) \,,
\]
so that we rewrite the previous equation as
\begin{multline*}
P \{ Y_{1}^{\prime} = y_{1}^{\prime}, Y_{2}^{\prime} = y_{2}^{\prime}, \dots, Y_{r-1}^{\prime} = y_{r-1}^{\prime} \} = \\ =\frac{\tilde{\varphi}_{1} (\mathbf{y}^{\prime})! \tilde{\varphi}_{2} (\mathbf{y}^{\prime})! \cdots \tilde{\varphi}_{n}(\mathbf{y}^{\prime})!}{(r-1)!} \sum_{h=1}^{n} \frac{\tilde{\varphi}_{h} (\mathbf{y}^{\prime})+1}{r}  P \{ \mathbf{X} = \boldsymbol{\tilde{\varphi}}^{(h)}(\mathbf{y}^{\prime}) \}.
\end{multline*}
We can now use \eqref{YvsX} and get
\begin{align*}
P \{ & Y_{1}^{\prime} = y_{1}^{\prime}, Y_{2}^{\prime} = y_{2}^{\prime}, \dots, Y_{r-1}^{\prime} = y_{r-1}^{\prime} \} = \\
&\\
& =\frac{\tilde{\varphi}_{1} (\mathbf{y}^{\prime})! \tilde{\varphi}_{2}(\mathbf{y}^{\prime})! \cdots \tilde{\varphi}_{n}(\mathbf{y}^{\prime})!}{( r-1)!} \sum_{h=1}^{n} \frac{\tilde{\varphi}_{h} (\mathbf{y}^{\prime})+1}{r} \, \frac{r!}{\tilde{\varphi}_{1} (\mathbf{y}^{\prime})! \cdots \left(\tilde{\varphi}_{h}(\mathbf{y}^{\prime})+1 \right)! \cdots \tilde{\varphi}_{n}(\mathbf{y}^{\prime})!} \times \\
& \qquad \times P \left\{ Y_{1} = \psi_{1} \left( \boldsymbol{\tilde{\varphi}}^{(h)}(\mathbf{y}^{\prime}) \right), \dots, Y_{r-1} = \psi_{r-1} \left( \boldsymbol{\tilde{\varphi}}^{(h)}(\mathbf{y}^{\prime}) \right), Y_{r} = \psi_{r} \left( \boldsymbol{\tilde{\varphi}}^{(h)}(\mathbf{y}^{\prime}) \right) \right\} \\
&\\
& = \sum_{h=1}^{n} P \left\{ Y_{1} = \psi_{1} \left( \boldsymbol{\tilde{\varphi}}^{(h)}(\mathbf{y}^{\prime}) \right), \dots, Y_{r-1} = \psi_{r-1} \left( \boldsymbol{\tilde{\varphi}}^{(h)}(\mathbf{y}^{\prime}) \right), Y_{r} = \psi_{r} \left( \boldsymbol{\tilde{\varphi}}^{(h)}(\mathbf{y}^{\prime}) \right) \right\} \,.
\end{align*}
%By means of some more remarks and manipulations we can obtain
Let us focus on the vector
\[
\left( \psi_{1} \left( \boldsymbol{\tilde{\varphi}}^{(h)}(\mathbf{y}^{\prime}) \right), \dots, \psi_{r-1} \left( \boldsymbol{\tilde{\varphi}}^{(h)}(\mathbf{y}^{\prime}) \right), \psi_{r} \left( \boldsymbol{\tilde{\varphi}}^{(h)}(\mathbf{y}^{\prime}) \right) \right) \,.
\]
Notice that $(r-1)$ components of this vector are respectively equal to $y_1', \dots, y_{r-1}'$ and the remaining component must necessarily be equal to $h$.\\
Thus, by recalling the exchangeability property of $(Y_1, \dots, Y_r)$, we get the equality
\[
P \{ Y_{1}^{\prime} = y_{1}^{\prime}, Y_{2}^{\prime} = y_{2}^{\prime}, \dots, Y_{r-1}^{\prime} = y_{r-1}^{\prime} \} = \sum_{h=1}^{n} P \{ Y_{1} = y_{1}^{\prime}, \dots, Y_{r-1} = y_{r-1}^{\prime}, Y_{r}=h \},
\]
whence the conclusion follows.
\end{proof}

As a remarkable property of the class of the EOM's one can prove its closure under all the transformations $\mathcal{K}_{1}$, $\mathcal{K}_{2}$ and $\mathcal{K}_{n,s}^{(N,r)}$.

\begin{proposition}\label{Trasf1}
Let $(X_1,\dots,X_N)$ be an EOM on $A_{N,r}$ and fix $n < N$. Conditionally on the event $\lbrace S_n=s\rbrace$, the variables $X_1,\dots, X_n$ are distributed according to an EOM on $A_{n,s}$.
\end{proposition}
\begin{proof}
It must be proved that conditionally on $\lbrace S_n=s \rbrace$, the vector $(X_1,\dots,X_n)$ is exchangeable, i.e.
$$P\left\{ X_{\sigma(1)}=x_1,\dots,X_{\sigma(n)}=x_n|S_n=s \right\}=P \left\{ X_1=x_1,\dots,X_n=x_n|S_n=s\right\}$$
for every permutation $\sigma$ of $\lbrace 1,2,\dots,n\rbrace$. It straightforwardly follows from the exchangeability property of the random variables $X_1,\dots,X_N$. Indeed, for any $(x_1, \dots, x_n) \in A_{n,s}$, we obtain
\begin{allowdisplaybreaks}
\begin{align*}
P \big\{ X_{\sigma(1)} &= x_1, \dots, X_{\sigma(n)} = x_n|S_n=s \big\} = \\
&\\
&=\frac{P\{X_1=x_{\sigma(1)},\dots,X_{\sigma(n)}=x_n\}}{P\{S_n=s\}} \\ %\qquad {\bf x}\in A_{n,s} \\
&\\
&=\sum_{\boldsymbol{\eta} \in A_{N-n,r-s}}\frac{P\{X_1=x_{\sigma(1)},\dots,X_{\sigma(n)}=x_n,X_{n+1}=\eta_1,\dots,X_{N}=\eta_{N-n}\}}{P\{S_n=s\}} \\ %\qquad{\bf x}\in A_{n,s}\\
&\\
&=\sum_{\boldsymbol{\eta} \in A_{N-n,r-s}}\frac{P\{X_1=x_1,\dots,X_n=x_n, X_{n+1}=\eta_1,\dots, X_{N}=\eta_{N-n}\}}{P\{S_n=s\}} \\ %\qquad {\bf x}\in A_{n,s}\\
&\\
&=\frac{P\{X_1=x_1,\dots,X_n=x_n,S_n=s\}}{P\{S_n=s\}}\\
&\\
&=P\{X_1=x_1,\dots,X_n=x_n|S_n=s\}
\end{align*}
\end{allowdisplaybreaks}\\[-0.8cm]
and this concludes the proof.
\end{proof}

The following Proposition can be easily derived from Lemma \ref{Lmm:elimination}. 

\begin{proposition}\label{Trasf2}
Let $(X_1,\dots,X_n)$ be an EOM on $A_{n,r}$ and $(X'_1, \dots, X'_n)$ be the occupancy numbers obtained by applying the transformation $\mathcal{K}_1$. Then, $(X'_1, \dots, X'_n)$ is an EOM on $A_{n,r-1}$.
\end{proposition}

The following Proposition can be easily obtained by taking into account equation \eqref{ProbTransfK2}.

\begin{proposition}\label{Trasf3}
Let $(X_1,\dots,X_n)$ be an EOM on $A_{n,r}$ and $(X''_1, \dots, X''_n)$ be the occupancy numbers obtained by applying the transformation $\mathcal{K}_2$. Then, $(X''_1, \dots, X''_n)$ is an EOM on $A_{n-1,r}$.
\end{proposition}

\begin{remark}
The transformation $\mathcal{K}_{2}$ can be seen as a special case of a more general class of transformations: we can consider the case where the $n$-th cell is eliminated and the $X_{n}$ particles that had fallen within it are distributed within the remaining cells according to an exchangeable occupancy model. The class of the EOM's is closed also w.r.t. this class of transformations. 
\end{remark}

From Proposition \ref{Trasf1} we know that a subvector $(X_1,\dots,X_n)$ of an EOM $(X_1,\dots,X_N)$ is still an EOM conditionally on a fixed value for the sum $S_{n}$ (closure property w.r.t. transformations of the type $\mathcal{K}_{n,s}^{(N,r)}$). Then we know that, by applying this transformation to an occupancy model of the form $\mathcal{M}^{(a)}$, we surely obtain an exchangeable model. It is then natural to wonder whether the latter is still of the form $\mathcal{M}^{(a)}$ (i.e., whether the closure property also holds for the class $\mathcal{M}^{(a)}$). The following proposition answers this question.  

\begin{proposition}
Let $(X_1,\dots,X_N)$ be a $\mathcal{M}_{N,r}^{(a)}$-model and define $S_n=X_1+\cdots+X_n$ with $n\leq N$. Conditionally on the event $\lbrace S_n=s\rbrace$, the variables $X_1,\cdots, X_n$ are distributed as a $\mathcal{M}_{n,s}^{(a)}$-model.
\end{proposition}
\begin{proof}
We need to prove that conditionally on $\lbrace S_n=s \rbrace$, the vector $(X_1,\dots,X_n)$ is distributed as a $\mathcal{M}_{n,s}^{(a)}$-model, i.e.
$$P\{X_1=x_1,\dots,X_n=x_n|S_n=s\}=\frac{\prod_{j=1}^{n} a(x_j)}{C_{n,s}^{(a)}}.$$
For any $(x_1,\dots,x_n) \in A_{n,s}$, we have
\begin{align*}
P \{ X_1 = x_1&, \dots, X_n = x_n|S_n = s \} =%\\
%&\\
%&= \frac{P\{X_1=x_1,\dots,X_n=x_n\}}{P\{S_n=s\}} \\ %\qquad {\bf x}\in A_{n,s}$$
%&\\
%&= \sum_{\boldsymbol{\eta} \in A_{N-n,r-s}}\frac{P\{X_1=x_1,\dots,X_n=x_n, X_{n+1}=\eta_1,\dots, X_{N}=\eta_{N-n}\}}{P\{S_n=s\}} \\ %\qquad {\bf x}\in A_{n,s}\\
%&\\
%&=
\sum_{\boldsymbol{\eta} \in A_{N-n,r-s}}\frac{\prod_{j=1}^n a(x_j)\prod_{i=1}^{N-n}a(\eta_i)}{C^{(a)}_{N,r} \, P\{S_n=s\}} %\\ %\qquad {\bf x}\in A_{n,s}\\
%&\\
%&
=K \prod_{j=1}^{n} a(x_j)
\end{align*}
where $$K := \frac{1}{C^{(a)}_{N,r} \, P\{S_n=s\}} \: \sum_{\boldsymbol{\eta} \in A_{N-n,r-s}} \: \prod_{i=1}^{N-n}a(\eta_i) \,.$$
Since $P\{X_1=x_1,\dots,X_n=x_n|S_n=s\}$ is a probability distribution, $K$ is exactly equal to $\left( C_{n,s}^{(a)} \right)^{-1}$ and this concludes the proof.
\end{proof}

As shown in Proposition \ref{Trasf2}, the action of dropping one particle at random (transformation $\mathcal{K}_1$) preserves the exchangeability property. As we will see in next Proposition, the structure of $\mathcal{M}^{(a)}$-model is preserved only under the technical assumption \eqref{condEOM}. It is easy to see that MB, BE, FD and the pseudo-contagious occupancy models satisfy this condition.

\begin{proposition}
Let $(X_{1}, X_{2}, \dots, X_{n})$ be a $\mathcal{M}_{n,r}^{(a)}$-model and let the function \mbox{$a:\{0,1,\dots\} \longrightarrow ]0,+\infty[$} satisfy the condition
\begin{equation}\label{condEOM}
\frac{C_{n,r-1}^{(a)}}{C_{n,r}^{(a)}}\sum_{h=1}^{n} \frac{x'_{h}+1}{r}\frac{a(x'_{h}+1)}{a(x'_h)}=1 \,, \quad \mbox{ for any } {\bf x'} \in A_{n,r-1} \,.
\end{equation}
Then, the occupancy vector $(X_1',\dots,X'_n)$, obtained by applying the transformation $\mathcal{K}_1$, is distributed according to the model $\mathcal{M}_{n,r-1}^{(a)}$.
\end{proposition}
%\begin{proof}
%We have to show that
%$$P\{X'_1=x'_1,\dots,X'_n=x'_n\} = \frac{\prod_{j=1}^n a(x'_j)}{C^{(a)}_{n,r-1}} .$$
%In view of \eqref{X'} the joint distribution of $(X'_1,\dots,X'_n)$ is given by
%\begin{align*}
%P\{X'_1=x'_1, \dots , X'_n=x'_n\} &= \sum_{h=1}^{n} \frac{x_{h}^{\prime}+1}{r} P \{ X_{1} = x_{1}^{\prime}, \dots, X_{h} = x_{h}^{\prime}+1, \dots, X_{n} = x_{n}^{\prime} \}\\
%&=\sum_{h=1}^{n} \frac{x_{h}^{\prime}+1}{r}\frac{a(x'_h+1)}{C_{n,r}^{(a)}}\displaystyle{\prod_{\substack{i=1\\i\neq h}}^{n}}a(x'_i)\\
%&=\frac{\prod_{j=1}^n a(x'_j)}{C^{(a)}_{n,r-1}}\frac{C_{n,r-1}^{(a)}}{C_{n,r}^{(a)}}\sum_{h=1}^{n} \frac{x'_{h}+1}{r}\frac{a(x'_{h}+1)}{a(x'_h)}\\
%&=\frac{\prod_{j=1}^n a(x'_j)}{C^{(a)}_{n,r-1}}\,,
%\end{align*}
%where the last equality follows from \eqref{condEOM}.
%\end{proof}
For brevity's sake we omit the proof that can be obtained rather easily by recalling the formula \eqref{X'} and by applying \eqref{condEOM}.
\begin{remark}
One can easily realize that the class of the $\mathcal{M}^{(a)}$-models is strictly contained within the class of the EOM's.
Examples can be easily found, for instance, by starting from models in the class $\mathcal{M}^{(a)}$ and by applying the transformations $\mathcal{K}_{1}$ or $\mathcal{K}_{2}$.
\end{remark}

\section*{Acknowledgements}
Fabio Spizzichino has been partially supported in the frame of Research Projects ``Modelli e Algoritmi Stocastici: Convergenza ed Ottimizzazione'' 2009 and ``Modelli Stocastici e Applicazioni alla Fisica e alla Finanza'' 2010 of University La Sapienza, Rome. The research of Fabrizio Leisen has been partially supported by the Spanish Ministry of Science and Innovation through grant ECO2011-25706.
%% Reference list
%%
%% References should be in the following form (or the BibTeX file
%% apt.bst should be used):
%%
%% For a journal:
%% Surname, Initial (year). Title of paper. {\em Journal title}
%% {\bf Vol,} page--range.
%%
%% For a book:
%% Surname, Initial (year). {\em Book title}. Publisher, Address.
%%
%% Note the following example of a reference list.
%
%\begin{thebibliography}{99}
%\footnotesize
%
%\bibitem{ref1}
%{\sc Ball, K. and Chain, H.} (1988). {\em Kurtosis: A Critical
%Review}, 2nd~edn. John Wiley, New York.
%
%\bibitem{ref2}
%{\sc Boyd, W.} (1978). Hyperbolic distributions. Doctoral Thesis,
%University of Boston School of Mathematics.
%
%\bibitem{ref3}
%{\sc Sichel, H.~S., Kleingeld, W.~J. and Assibey-Bonsu, W.}
%(1992).  A comparative study of three frequency-distribution
%models for use in ore valuation. {\em J. S. Afr. Inst. Min. Met.}
%{\bf 92,} 91--99.
%
%\end{thebibliography}

\bibliographystyle{plainnat}

%=============================================================================

%=============================================================================
%\appendix

%=============================================================================
%\section{Proofs}
%=============================================================================

%=============================================================================

\end{document}